\documentclass[12pt, reqno]{amsart}  
\usepackage{amsmath,amsthm,amssymb,amsfonts,eucal,fullpage,latexsym,mathrsfs,stmaryrd}
\usepackage{enumerate}
\usepackage[dvipsnames]{xcolor}
\usepackage{lipsum}
\usepackage{array}
\usepackage{lineno}

\theoremstyle{plain}

\numberwithin{equation}{subsection}
\theoremstyle{plain}
\newtheorem{Proposition}[equation]{Proposition}

\newtheorem{Theorem}[equation]{Theorem}
\newtheorem{Lemma}[equation]{Lemma}
\theoremstyle{definition}
\newtheorem{Definition}[equation]{Definition}

\newtheorem{Example}[equation]{Example}
\newtheorem{Remark}[equation]{Remark}

\def\phi{\varphi}

\renewcommand{\leq}{\leqslant}
\renewcommand{\geq}{\geqslant}

\newcommand{\D}{\mathbb{D}}
\newcommand{\T}{\mathbb{T}}
\newcommand{\Pol}{\mathscr{P}}

\begin{document}



\title{More Properties of Optimal Polynomial Approximants in Hardy Spaces}
\author[Cheng]{Raymond Cheng}
\address{Department of Mathematics and Statistics, Old Dominion University, Norfolk, VA 23529, USA. } \email{rcheng@odu.edu}
\author[Felder]{Christopher Felder}
\address{Department of Mathematics, Indiana University, Bloomington, IN 47405, USA. } \email{cfelder@iu.edu}

\date{\today}

\begin{abstract}
This work studies optimal polynomial approximants (OPAs) in the classical Hardy spaces on the unit disk, $H^p$ ($1 < p < \infty$). For fixed $f\in H^p$ and $n\in\mathbb{N}$, the OPA of degree $n$ associated to $f$ is the polynomial which minimizes the quantity $\|qf-1\|_p$
over all complex polynomials $q$ of degree less than or equal to $n$.
We begin with some examples which illustrate, when $p\neq2$, how the Banach space geometry makes these problems interesting. We then weave through various results concerning limits and roots of these polynomials, including results which show that OPAs can be witnessed as solutions of certain fixed point problems. 
Finally, using duality arguments, we provide several bounds concerning the error incurred in the OPA approximation. 
\end{abstract}

\maketitle


\tableofcontents


\section{Introduction}
This paper concerns a minimization problem in classical Hardy spaces on the unit disk $\mathbb{D}$,
\[
H^p := \left\{ f\in \operatorname{Hol}(\mathbb{D}) : \sup_{0 \le r < 1} \int_0^{2\pi} \left|f(re^{i\theta})\right|^p \, d\theta < \infty \right\},
\]
where  $\operatorname{Hol}(\D)$ denotes the collection of holomorphic functions on $\D$. 
As is standard, for $1\le p < \infty$, we denote the norm of $f \in H^p$ as 
\[
\|f\|_p := \left(\sup_{0 \le r < 1} \int_0^{2\pi} \left|f(re^{i\theta})\right|^p \, d\theta\right)^{1/p}.
\]
When $p=\infty$, we have the set of bounded analytic functions
\[
H^\infty := \left\{f\in \operatorname{Hol}(\mathbb{D}) : \sup_{z \in \D}|f(z)| < \infty \right\},
\]
with corresponding norm
\[
\|f\|_\infty := \sup_{z \in \D}|f(z)|.
\]
We will frequently view these spaces as subspaces of the Lebesgue spaces $L^p := L^p(\T, dm)$, where $dm$ is normalized Lebesgue measure on the unit circle $\T$.

Our main objects of study in this paper are \textit{optimal polynomial approximants} (OPAs) in Hardy spaces; these are solutions to minimization problem
\[
\inf_{q \in \Pol_n}\|qf-1\|_p, 
\]
where $f \in H^p$ and $\Pol_n$ is the set of complex polynomials of degree less than or equal to $n$. We point out that the infimum above is actually a minimum. For in a uniformly convex Banach space, any closed subspace enjoys a unique nearest point property. In our context, this means the problem of finding a degree $n$ OPA can be restated as finding the solution to 
\[
\inf_{h \in f\Pol_n}\|h - 1\|_p,
\]
which is given by the \textit{metric projection} of $1$ on the subspace $f\Pol_n$. 
A priori, the minimizing argument may not be unique. However, when $1<p<\infty$, it is well-known that there is, in fact, a unique minimizing polynomial (due to the uniform convexity of the space). 
When $p\neq 2$, the projection is \textit{non-linear}, which starkly contrasts the Hilbert space setting. 

For $p\neq2$, this problem was originally studied by Centner \cite{Cent}, and considered again in an additional paper by Centner and the authors \cite{CCF}. We will give some background now, but point the reader to \cite{Cent, CCF} for more thorough exposition, and to \cite{BCLSS1, BCLSS2, BFKSS, BKLSS, BKKLSS, Cent, FMS, SS2, SS1} for relevant work on OPAs in various Hilbert spaces.

When $p=2$, this problem was first studied by engineers in work related to digital filter design. The problem reemerged later as a potential way to study cyclic vectors for the forward shift (see \cite{BC} for historical discussion). This renewed interest is evidenced by many papers over the last decade (again, see \cite{BC}, as well as \cite{MR4396661, arora2022optimal} for recent results in the weighted and non-commutative settings, respectively). Other than the work in \cite{CRSX}, these results concern only Hilbert spaces, where the geometry makes computation of OPAs an explicit (but non-trivial!) linear algebra exercise. For example, in $H^2$, the coefficients (say, $a_0, \ldots, a_n$) of the OPA of degree $n$ associated to a function $f \in H^2$ can be found via the linear system
\begin{equation}\label{opt-sys}
\left(\langle S^jf, S^kf \rangle_{H^2}\right)_{0 \le j,k \le n} (a_0, \ldots, a_n)^T = \left(\overline{f(0)}, 0, \ldots, 0 \right)^T,
\end{equation}
where $S$ is the forward shift operator, given by $f(z) \mapsto zf(z)$ (see, e.g., \cite[Theorem~2.1]{FMS}). 

In the Banach space setting (e.g. $H^p$, $p \neq 2$), there is not a direct analogue of this exercise, and the non-linearity of the metric projection makes explicit calculation of OPAs a highly non-trivial task. In the next section, we will state precisely the definition of optimal polynomial approximant. Before moving there, let us give an outline of the paper:
\begin{itemize}
\item Section \ref{prelim} will formally introduce the OPA problem, give some background information concerning the geometry of Banach spaces, and provide some examples illustrating how this geometry differs from that of Hilbert space. 
\item The results of Section \ref{cont} are broken into three parts:
    \begin{itemize}
        \item Convergence of OPAs under variance of the parameters of the OPA problem (e.g., $n, p$, and $f$). 
        \item The location of roots of OPAs.
        \item Constant and linear OPAs as solutions to a fixed point problem. 
    \end{itemize}
\item Using duality, Section \ref{error} establishes various bounds for the error $\|qf - 1\|_p$.
\end{itemize}


\section{Preliminaries \& Geometric Oddities}\label{prelim}

We begin here by providing some background material concerning the geometry of Banach spaces, followed by several examples in $H^p$ which illustrate some oddities that arise when $p \neq 2$.

Let $\mathbf{x}$ and $\mathbf{y}$ be vectors belonging to a normed linear space $\mathscr{X}$.  We say that $\mathbf{x}$ is {\em orthogonal} to $\mathbf{y}$ in the Birkhoff-James sense \cite{AMW, Jam}  if
\begin{equation}\label{2837eiywufh[wpofjk}
      \|  \mathbf{x} + \beta \mathbf{y} \|_{\mathscr{X}} \geq \|\mathbf{x}\|_{\mathscr{X}}
\end{equation}
for all scalars $\beta$. 
In this situation we write $\mathbf{x} \perp_{\mathscr{X}} \mathbf{y}$. 
In the case $\mathscr{X} = L^p$,  let us write $\perp_p$ instead of  $\perp_{L^p}$, and similarly for $\mathscr{X} = H^p$.  
 
For $1<p<\infty$, there is also a function-theoretic test for $p$-orthogonality, which we note now. 

\begin{Theorem}[James \cite{Jam}]
Suppose $1<p<\infty$.  Then for $f$ and $g$ belonging to $L^p$, we have
\begin{equation}\label{BJp}
 {f} \ \perp_{p} \ {g}  \iff   \int_\T |f|^{p - 2} \overline{f} g \,dm  = 0,
\end{equation}
where any occurrence of ``$|0|^{p - 2} 0$'' in the integrand is interpreted as zero.
\end{Theorem}

In light of \eqref{BJp} we adopt, for a measurable function $f$ and any $s > 0$, the notation
\begin{equation}\label{definition-de-zs}
f^{\langle s \rangle} := |f|^{s-1}\overline{f}.
\end{equation}

If $f \in L^p$, then $f^{\langle p - 1\rangle}\in L^q$, where $q$ is the classical H\"{o}lder conjugate to $p$, satisfying $\frac1p + \frac1q = 1$.
For $g \in L^p$ and $f \in L^q$, we use the standard notation for the dual pairing
\[
\langle f, g \rangle = \int_\T f \overline{g} \, dm,
\]
and from \eqref{BJp}, we have
\begin{equation}\label{ppsduwebxzrweq5}
f \ \perp_p \ g \iff \langle g, f^{\langle p - 1\rangle} \rangle = 0.
\end{equation}
Consequently, the relation $\perp_{p}$ is linear in its second argument when $1<p<\infty$, and it then makes sense to speak of a vector being orthogonal to a subspace. We use this now to formally define OPAs.
\begin{Definition}[OPA]
Let $1<p<\infty$ and let $f \in H^p$. Given a non-negative integer $n$, the \textit{$n$-th optimal polynomial approximant} to $1/f$ in $H^p$ is the polynomial solving the minimization problem
\[
\min_{q \in \Pol_n}\|qf - 1\|_p,
\]
where $\Pol_n$ is the set of complex polynomials of degree less than or equal to $n$. This polynomial exists, is unique, and will be denoted by
\[
q_{n,p}[f].
\]
\end{Definition}

Given previous discussion on the metric projection, it is immediate that 
\[
1 - q_{n,p}[f] f \  \perp_p  \ \bigvee\{ f, zf, z^2 f,\ldots, z^n f \}.
\]
We note that we will use the notations $z^kf$ and $S^kf$ interchangeably when there is no risk of confusion. 
We will also use the notation $[f]_p$ to denote the closure of $\bigvee\{ f, zf, z^2 f, z^3f, \ldots \}$ in $H^p$, i.e., 
\[
[f]_p := \overline{\bigvee\{ f, zf, z^2 f, z^3f, \ldots \} }^{\ H^p}.
\]
In order to avoid trivialities, we will also often ask that $f(0) \neq 0$; in the case that $f(0) = 0$, this is equivalent to $1 \ \perp_p f$, and so the metric projection of 1 onto $f\Pol_n$ is identically zero.

In connection with Birkhoff-James orthogonality, there is a version of the Pythagorean Theorem for $L^p$.  
This theorem takes the form of a family of inequalities relating the lengths of orthogonal vectors with that of their sum \cite[Corollary 3.4]{CR}.

\begin{Theorem}\label{pythagthm}
Suppose that $x\ \perp_p \ y$ in $L^p$.
If $p \in (1, 2]$, then
\begin{align*}
   \| x + y \|^p_p & \leq  \|x\|^p_p + \frac{1}{2^{p-1}-1}\|y\|^p_p \\
    \| x + y \|^2_p & \geq  \|x\|^2_p + (p-1)\|y\|^2_p.
\end{align*}
If $p \in [2, \infty)$, then
\begin{align*}
   \| x + y \|^p_p & \geq  \|x\|^p_p + \frac{1}{2^{p-1}-1}\|y\|^p_p \\
    \| x + y \|^2_p & \leq  \|x\|^2_p + (p-1)\|y\|^2_p.
\end{align*}
\end{Theorem}

These Pythagorean inequalities enable us to obtain bounds and estimates when $p\neq 2$, in lieu of exact calculations possible in the Hilbert space case.


The following examples illustrate some of the ways the geometry of $H^p$ ($p\neq 2$) can run counter-intuitive to experience in Hilbert space. Although these examples may not be immediately surprising to the Banach space enthusiast, we relay them for the general functional analyst, especially working in linear approximation problems, as interesting observations related to natural geometric questions.

\begin{Example}
In a Hilbert space, an orthogonal projection is always a contraction. However, when $p\neq 2$, the norm of the metric projection of a vector can exceed the length of the vector itself.  

Consider the linear OPA for $f(z) = 1 + 0.5z$ in $H^4$. Numerically, we find that $Q(z) := q_{1,4}[f] \approx 0.9771018  -  0.4339644z$, and thus
\[
     \|Qf\|_4^4 = 1.10294 > 1.
\]
\end{Example}

\begin{Example}
In $H^2$, it is simple to verify that if $F(0) = 0$, then, for $c\in \mathbb{C}$, the quantity
\[
\| c + F(z) \|_2
\]
is minimized when $c =0$. However, this is not the case when $p \neq 2$. 

For example, let $p=4$ and $F(z) = z + 2z^2$. Then
\[
      \| c + F \|_4^4  =  33 +8c +  20 c^2 + c^4.
\]
Numerically, this is minimized when $c \approx -0.199209$.  In particular, the value of the minimizing argument can be nonzero when $p \neq 2$.
\end{Example}

\begin{Example}
Notice for $f \in H^2$ and any $n > 0$, 
we have 
\[
1 - q_{n, 2}[f]f \ \perp_2 \ zf \ \ \ \text{and} \ \ \ 1 \ \perp_2 \ zf.
\]
Using linearity, we have
\[
q_{n,2}[f]f = 1 - (1 - q_{n,2}[f]f)\ \perp_2 \ zf.
\]
It is natural to ask if this is true when $p\neq 2$, i.e., is it true in general that 
\[
q_{n,p}[f]f \ \perp_p \ zf?
\]

Let us take $p=4$, $n=1$, and $f(z) = 1 + 2z + z^8$.
Numerically, one can find that
\[
       \int_{\mathbb{T}} (q_{1,p}[f]f)^{\langle p-1\rangle}zf\,dm \approx 0.00355837,
\]
which is nonzero, and so orthogonality fails. This illustrates that for $p\neq 2$, the relation $\perp_p$ fails to be linear in its first argument, and so $q_{n,p}[f]f$ is not necessarily orthogonal to $zf$.
\end{Example}

\begin{Example}
For $f, g \in H^2$, an exercise shows that if $\|f\|_2  \leq  \|g\|_2$, then $\|1 + zf\|_2  \leq \|1 + zg\|_2$. Might a similar statement hold for $p \neq 2$? 
The following example shows that the answer is \textit{no}.

Let $p = 4$ and choose
\begin{align*}
   f(z) &= (0.9)(1+ z + z^2) \\
   g(z) &= -1-z-z^2.
\end{align*}

It is immediate that $\|f\|_4 < \|g\|_4$.
However, numerically, we find
\begin{align*}
     \|1 + zf(z)\|_4^4 &\approx  31.9339\\
     \|1 + zg(z)\|_4^4  &\approx 20.0000.
\end{align*}
\end{Example}

With these examples in hand, it may now be reasonable to suspect that OPAs have a dependence on $p$ which is  highly non-linear. In general, this is true. Let us demonstrate this with what we describe as the OPA ``error"-- the quantity $\|q_{n,p}[f]f - 1\|_p$. We use this as motivation in Section \ref{cont}, where we study the $p$-dependence of OPAs. 

\begin{Example}
For $c>0$ and $m$ a positive integer, consider $f(z) = 1 + cz^m$.  Let us show that when $c=2$, we have
\[
      \|q_{0,2}[f]f-1\|_2  \neq  \|q_{0,4}[f]f-1\|_4.
\]
For $p=2$, 
\begin{align*}
     \|af-1\|_2^2 &=  \int_0^{2\pi} \left([a-1]^2 + ace^{im\theta}\right) \left([a-1]^2 + ace^{-im\theta}\right)\,\frac{d\theta}{2\pi} \\
     &= (a-1)^2 + a^2c^2.
\end{align*}
Notice that the result of the integration is the extraction of the constant OPA. Minimizing this expression (by differentiating with respect to $a$), we find
\[
a = \frac{1}{1+c^2}.
\]
This yields
\begin{align*}
    \|af-1\|_2^2 &= \left(\frac{1}{1+c^2} - 1 \right)^2 + \frac{c^2}{(1 + c^2)^2} \\
      &=  \frac{c^2}{1 + c^2}.
\end{align*}

This is equal to $4/5$ when $c=2$, so we obtain
\[
     \|af-1\|_2  =  \sqrt{4/5} \approx 0.894427.
\]

For $p=4$, we have
\[
     \| af - 1 \|_4^4 = \int_0^{2\pi} \left(a+ac e^{im\theta} - 1\right)^2 \left(a+ace^{-im\theta}-1\right)^2\,\frac{d\theta}{2\pi},
\]
and one may extract the constant term as
\[
     (a-1)^4 + 4a^2c^2(a-1)^2 +a^4c^4.
\]

Next, setting $c=2$, one may numerically find that $a \approx 0.121991$ minimizes the above expression.  

Finally, this yields
\[
     \|q_{0,4}[f]f-1\|_4 \approx  \sqrt[4]{0.781388} \approx 0.940192.
\]
\end{Example}

In addition to the error, one may also notice that the OPAs themselves vary with $p$.
For example, letting $f(z) = 1 + 2z + z^8$, one may numerically find that 
\begin{align*}
      q_{0,4}[f] &\approx  0.0970262\\
      q_{0,6}[f] &\approx  0.0674066.
\end{align*}

However, this is not always the case(!). 
The following example is a generalization of \cite[Example~6.1]{Cent}, which showed that the constant OPAs for $f(z) = 1 - z$ do not vary with $p$. 

\begin{Example}
Let $1<p<\infty$ and let $f \in H^p$. Let $\lambda \in \mathbb{C}$ and let $h = 1 + f$.
Suppose that $|f(e^{it})| = 1$ a.e. on $\mathbb{T}$ and $\overline{f(e^{-it})} = f(e^{it})$ (i.e., the Fourier coefficients of $f$ are real). 

Putting $a = q_{0,p}[h]$, we observe:
\begin{align*}
\| a(1 + f(e^{it})) - 1 \|_p &= \| a f(e^{it}) +  a - 1 \|_p \\
&= \| a +  \overline{f(e^{it})} (a - 1) \|_p \ \ \ \text{(multiply by $\overline{f}$, inner)}\\
&= \| a +  f(e^{-it}) (a - 1) \|_p \ \ \ \text{(real coefficients)}\\
&= \| a +  f(e^{it}) (a - 1) \|_p \ \ \ \text{($t \mapsto -t$)}\\
&= \| - a + 1 - 1 -  f(e^{it}) (a - 1) \|_p \ \ \ \text{(multiply by -1 and add 0)}\\ 
&= \|(1 - a)(1+ f(e^{it})) -1 \|_p.
\end{align*}
This tells us that $a = \frac12$, which is independent of $p$.
Note that if $f$ is any Blaschke product with real zeros, the hypotheses above are satisfied.

\end{Example}


\section{Limits and Continuity}\label{cont}

In this section, we provide results which relate to  varying the parameters in the OPA problem (i.e., the degree $n$, the value of $p$, and the function $f$). We first deal with this directly. Then, as a corollary, the first subsection below discusses the possible set of roots for OPAs. In the final subsection, we show that OPAs (in certain cases) are solutions to a fixed-point problem. 
All of these results enable us to make estimates concerning OPAs, knowing that exact computation is difficult when $p\neq 2$.

We begin by recording, without proof, a known result about metric projections.

\begin{Proposition}
Let $1<p<\infty$.  Let $f \in H^p$ with $f(0) \neq 0$ and let $h$ be the metric projection of 1 onto $[f]_p$. Then, in norm, 
\[
q_{n,p}[f] f \longrightarrow h \ \text{ as } \ n \longrightarrow \infty.
\]
\end{Proposition}
This result can be seen as a consequence of the fact that as $n \longrightarrow \infty$, the metric projections from $H^p$ onto $f\Pol_n$ converge, in the strong operator topology, to the metric projection from $H^p$ onto $[f]_p$ (see, e.g., \cite[Proposition~4.8.3]{CMR}).

In the following proposition, for $1<p<\infty$ and $f \in H^p$, we write $Q_{\infty} f$ for the metric projection of $1$ onto $[f]_p$, understanding that $Q_{\infty}$ need not be a bonafide $H^p$ function. The next result tells us something about the error incurred by approximating $q_{n,p}[f]$ using the Taylor polynomials of $Q_{\infty}$, when the (rather strict) assumption of norm convergence holds.

\begin{Proposition}\label{trunk}
     Let $1<p<\infty$, and $f \in H^p$.  Suppose that the representation
     \[
       Q_{\infty}(z) f(z)   = \sum_{k=0}^{\infty} \alpha_k z^k f(z) 
\]
converges in norm.  Then there exist a positive constant $C$ and an index $N$ such that
\[
       \| q_{n,p}[f] f - Q_{(n)}f\|_p^r  \leq C \| Q_{\infty} f - Q_{(n)} f \|_p
\]
for all $n \geq N$, where $Q_{(n)}(z) = \sum_{k=0}^{n} \alpha_k z^k$, and $r$ and $K$ are the applicable Pythagorean parameters.
\end{Proposition}

\begin{proof}
From the orthogonality relation
\[
        1 -  q_{n,p}[f] f  \  \perp_p \  q_{n,p}[f]f - Q_{(n)} f,
\]
the Pythagorean inequality gives
\[
        \| 1 -  q_{n,p}[f] f  \|_p^r + K \| q_{n,p}[f] f - Q_{(n)} f \|_p^r  \leq  \| 1 - Q_{(n)} f\|_p^r.
\]

Rearrange and estimate to get
\begin{align*}
     K \| q_{n,p}[f] f - Q_{(n)} f \|_p^r  &\leq    \| 1 - Q_{(n)} f\|_p^r -  \| 1 -  q_{n,p}[f] f  \|_p^r \\
     &\leq    \| 1 - Q_{(n)} f\|_p^r -  \| 1 -  Q_{\infty} f  \|_p^r \\
     &\leq   r \| 1 - Q_{(n)} f\|_p^{r-1} \left( \| 1 - Q_{(n)} f\|_p - \| 1 -  Q_{\infty} f  \|_p  \right) \\
     &\leq   r \| 1 - Q_{(n)} f\|_p^{r-1}  \|  Q_{\infty}f -  Q_{(n)} f  \|_p  \\
     &\leq   2r \| 1 - Q_{\infty} f\|_p^{r-1}  \| Q_{\infty}f -  Q_{(n)}f  \|_p, 
\end{align*}
for $n$ sufficiently large.  In the third step we applied the elementary inequality 
\[
     a^r - b^r \leq ra^{r-1}(a-b),
\] 
for $0 < b < a$ and $r>1$

This verifies the claim, with $C = 2r \| 1 - Q_{\infty} f\|_p^{r-1}/K$.
\end{proof}

The previous proposition can be applied when $f$ is any polynomial; we record that result now. 

\begin{Proposition}
Suppose that $1<p<\infty$.
Let $z_1$, $z_2$,\ldots, $z_N$ be a sequence of nonzero points of $\mathbb{D}$, and define
\[
      f(z) :=  \left( 1 - \frac{z}{z_1} \right) \left( 1 - \frac{z}{z_2}   \right)\cdots \left( 1 - \frac{z}{z_N}   \right).
\]
Set $r=2$ if $1<p\leq 2$, and set $r=p$ if $2< p<\infty$.
Then the metric projection $h$ of the unit constant function 1 onto the subspace $[f]_p$ of $H^p$ has a norm convergent representation
\[
     h(z) =  \sum_{k=0}^{\infty} b_k z^k f(z),
\]
and there exists a positive constant $C$ such that
\begin{equation}\label{erresttail}
         \left\|  q_{n,p}[f]f  -  \sum_{k=0}^{n} b_k z^k f\right\|_p^r  \leq  C \left\|  \sum_{k=n+1}^{\infty} b_k z^k f  \right\|_p
\end{equation}
for all positive integers $n$.
\end{Proposition}

We omit the proof here, but note that the metric projection $h$ must vanish at the zeros of $f$; in turn, boundedness of the difference-quotient operator, given by
\[
\left(B_wf\right)(z) = \frac{f(z) - f(w)}{z - w}, \ \ \ z, w\in \mathbb{D},
\]
(applied where $f(w) = 0$) then ensures the norm convergent representation.


\subsection{Continuity}

As discussed earlier, OPAs generally vary with $p$. We discuss this variance here, first showing that when $f$ is a bounded function, $q_{n,p}[f]$ varies continuously with respect to $p$.

\begin{Lemma}
Let $f \in H^{\infty}$ with $f(0) \neq 0$ and let $d \in \mathbb{N}$. If $(p_k)_k \subseteq (1, \infty)$ with $p_k \longrightarrow p \in (1, \infty)$, then $q_{d,p_k}[f]$ converges to $q_{d,p}[f]$ uniformly as $k \longrightarrow\infty$.
\end{Lemma}

\begin{proof}
    Let us write
    \begin{align*}
        f(z) &= \sum_{j=0}^{\infty} f_j z^j, \\
        q_{d,p}[f](z) &= a_0 + a_1 z + a_2 z^2 +\cdots+ a_d z^d, \\
         q_{d,p_k}[f](z) &= a_0^{(k)} + a_1^{(k)}  z + a_2^{(k)}  z^2 +\cdots+ a_d^{(k)}  z^d.
    \end{align*}
      
Since $f \in H^{\infty}$, Cauchy-Schwarz yields
\[
\left| \int_{\mathbb{T}} fz^{-j}\,dm \right|
\leq \|f\|_{\infty}
\]
for all $j \ge 0$. 
Hence, all of the coefficients $f_k$ are bounded by $\|f\|_{\infty}$. 
    
Letting $p_k'$ be the dual exponent of $p_k$, we observe
\[
       |a_0^{(k)}f_0-1| =   \left|  \int_{\mathbb{T}} (q_{d,p_k}[f]f-1)\,dm \right|  \leq \|q_{d,p_k}[f]f-1\|_{p_k} \|1\|_{p_k'} \leq 1,
\]     
and so the sequence $\{a_0^{(k)}\}$ is bounded.

Further, since
\[
    \left| a_0^{(k)} f_j + a_1^{(k)} f_{j-1} +\cdots + a_j^{(k)} f_0\right| \leq \left|  \int_{\mathbb{T}} (q_{d,p_k}[f]f-1)z^{-j}\,dm  \right| \leq 1,
\]    
it follows
\[
    |a_j^{(k)}| \leq  \frac{ \left| a_0^{(k)} f_j + a_1^{(k)} f_{j-1} +\cdots + a_{j-1}^{(k)} f_1 \right|+1 }{ |f_0| },
\]    
for all $k\in \mathbb{N}$ and $1\leq j \leq d$. That is, $\{a_j^{(k)}\}_{k=1}^{\infty}$ is also a bounded sequence for $1 \leq j \leq d$.  By passing to a subsequence and relabeling, we can assume that $\{q_{n,p_k}[f]\}_{k=1}^{\infty}$ is a uniformly convergent sequence of polynomials, which converges to some polynomial, say, $A(z) = a_0 + a_1 z+\cdots+a_d z^d  \in \Pol_d$.

Now, for $0 \leq j \leq d$, recall the orthogonality equations
\[
       \int_{\mathbb{T}} [q_{d,p_k}[f]f -1]^{\langle p_k -1\rangle} z^j f\, dm = 0.
\]
Taking $k\longrightarrow\infty$ and invoking uniform convergence, we find that  
\[
       \int_{\mathbb{T}} [Af -1]^{\langle p_k -1\rangle} z^j f\, dm = 0,
\]
for $0 \leq j \leq d$ (the taking of $\langle p_k-1\rangle$ powers also being well behaved).  By uniqueness of the optimal polynomial, it must be that
\[
       A(z) = q_{d,p}[f](z).
\]
    
Since every subsequence of the originally given sequence  $\{q_{d,p_k}[f]\}_{k=1}^{\infty}$ has a further subsequence that converges to the same limit $q_{d,p}[f]$, it must be that
\[
        q_{d,p_k}[f] \longrightarrow q_{d,p}[f]
\]
uniformly.  
\end{proof}

We now present another continuity result-- continuity in $f$. In particular, if $f_k \longrightarrow f$ in $H^p$, then $q_{n, p}[f_k] \longrightarrow q_{n, p}[f]$. 
Before establishing this result, we need a couple of lemmas.

\begin{Lemma}
     Let $1<p<\infty$ and $1/p + 1/q = 1$.  If $\phi_k \longrightarrow \phi$ in $L^p$, then $\phi_k^{\langle p-1\rangle}\longrightarrow \phi^{\langle p-1\rangle}$ in $L^q$.
\end{Lemma}

\begin{proof}
First, we check that
\[
      \int_{\mathbb{T}} \left|\phi^{\langle p-1\rangle}\right|^q\,dm =  \int_{\mathbb{T}} |\phi|^{(p-1)q}\,dm =  \int_{\mathbb{T}} |\phi|^p\,dm,
\]
and so $\phi^{\langle p-1\rangle} \in L^q$; similarly $\phi_k^{\langle p-1\rangle} \in L^q$.  

Next, we apply the generalized dominated convergence theorem, using the sequential bound
\[
    \left| \phi_k^{\langle p-1\rangle} - \phi^{\langle p-1\rangle} \right|^q  \leq  2^{q-1} \left(|\phi_k|^p + |\phi|^p\right) \ \mbox{a.e.-$dm$},
\]
with the Carleson-Hunt theorem supplying pointwise convergence almost everywhere.  The conclusion is
\[
      \int_{\mathbb{T}} \left|\phi_k^{\langle p-1\rangle} - \phi^{\langle p-1\rangle}\right|^q\,dm \longrightarrow 0,
\]
as claimed.
\end{proof}

Below, we use the standard notation $\hat{f}(n)$ to denote the $n$-th Fourier coefficient of a function $f \in f \in L^p$.  
\begin{Lemma}
     Let $1<p<\infty$.  If $\phi \in H^p$, then
     \[
          \|\phi\|_p^r \geq |\hat{\phi}(0)|^r + K |\hat{\phi}(1)|^r + K^2 |\hat{\phi}(2)|^r + \cdots,
     \]
     where $r$ and $K$ are the lower Pythagorean parameters.
\end{Lemma}

    \begin{proof}
 This follows immediately from the orthogonality relations
     \[
             z^k \   \perp_p   \ z^m H^p\ \ \ \ \forall m > k \geq 0,
     \]
and repeated application of the lower Pythagorean inequality.
\end{proof}

We are now prepared to prove the aforementioned result. 

\begin{Theorem}\label{cont-f}
Suppose that $1<p<\infty$ and $n\in \mathbb{N}$.
  Let  $f_k  \in H^p$ and let $Q_k  := q_{n,p}[f_k]$ for each $k \in \mathbb{N}$.  If $f_k \longrightarrow f$ in $H^p$, and $f(0) \neq 0$, then $Q_k \longrightarrow Q := q_{n,p}[f]$.
\end{Theorem}

\begin{proof}  Let us first handle the case $n=1$, and write $Q_k(z) = a_k + b_k z$ for $q_{1,p}[f]$.

Since $f_k(0) \longrightarrow f(0)$, and $f(0) \neq 0$, there is no harm in assuming that there exists $c>0$  such that $|f_k(0)| \geq c$ for all $k$.

     From the relation 
     \[
            1 \   \perp_p  \  zH^p
     \]
     we see that
     \[
           1  \geq \|1 - Q_k f_k\|_p^r  \geq  |1 - a_k f_k(0)|^r  +  K \| Q_k f_k - a_k f_k(0) \|^r_p,
     \]
     where $r$ and $K$ are the lower Pythagorean parameters.  It follows that
     \[
          1  \geq |1 - a_k f_k(0) |,
     \]
     implying that
     \[
           |a_k|  \leq \frac{2}{c}
     \]
     for all $k$.  Thus $\{a_k\}$ is a bounded complex sequence, from which we can extract a convergent subsequence, which for now we relabel as the original sequence.
     
     Next, subharmonicity and the triangle inequality yield
     \[
             c|b_k|  \leq |b_k f_k(0)| \leq   \| b_k z f_k \|_p  \leq \|1 - (a_k+ b_k z)f_k \|_p +  \|1 - a_k f_k \|_p.
     \]
     The last expression on the right side is uniformly bounded as $k$ varies through $\mathbb{N}$, and hence $\{b_k\}$ is a bounded sequence.  Once again we may draw a convergent subsequence, and relabel it so that
     \[
            Q_k  =  a_k + b_k z
     \]
     converges uniformly to some $R(z) =  a + bz$.
     
     It needs to be shown that $R = Q := q_{1,p}[f]$.  For this we rely on the elementary result that if $v_k \longrightarrow v$ in a Banach space and $\lambda_k \longrightarrow \lambda$ in its dual space, then $\lambda_k(v_k) \longrightarrow \lambda(v)$.  
     
     We apply this, identifying
     \begin{align*}
           v_k &= f_k\\
           v  &= f\\
           \lambda_k(\cdot)   &=   \int_{\mathbb{T}} \left( 1 - [a_k + b_kz]f_k \right)^{\langle p-1 \rangle} (\cdot)\,dm\\
            \lambda(\cdot)   &=   \int_{\mathbb{T}} \left( 1 - [a + bz]f \right)^{\langle p-1 \rangle} (\cdot)\,dm.
     \end{align*}
     
     Then Lemma 0.49 ensures that $\lambda_k \longrightarrow \lambda$, as needed.
     
     The conclusion is that $\lambda(v) = \lim_{k\rightarrow\infty}\lambda_k(v_k) = \lim_{k\rightarrow\infty} 0=0$, or
     \[
          1 - (a+bz)f \  \perp_p  \ f.
     \]
     
     Repeat this argument with the choices
     \[
           v_k = z f_k \ \ \ \mbox{and} \ \ \ v = zf
     \]
     to see that 
      \[
          1 - (a+bz)f \   \perp_p  \ zf
     \]
     as well.   This forces, $R(z) = Q(z) = a + bz = q_{1,p}[f](z)$, as claimed.  
     
     So far, we only know that there is a subsequence that satisfies the claim.  However, we see that every subsequence of the original sequence $\{f_k\}$ has a further subsequence for which the linear OPAs tend {\it to the same limit} $a + bz$, the linear OPA from $f$ being unique.  This proves that in fact the full sequence $\{a_k + b_k z\}$ converges to $a + bz$.
     
     This verifies the claim when $n=1$.
     
     More generally, for arbitrary $n \in \mathbb{N}$, let
     \[
           Q_k(z) = q_{n,p}[f](z) = a_0^{(k)} + a_1^{(k)} z + \cdots + a_n^{(k)} z^n.
     \]
     
     From Lemma 0.50, we get
     \begin{align*}
          1 &\geq \left\|1 - Q_k f_k \right\|_p^r \\
             &\geq  \left|1 - a_0^{(k)} f_0\right|^r  +  \sum_{m=1}^{\infty}  K^{m} \left| a_0^{(k)}f_m + a_1^{(k)}f_{m-1} + \cdots + a_m^{(k)}f_0   \right|^r\\
          \frac{1}{K^{m/r}} &\geq \left| a_0^{(k)}f_m + a_1^{(k)}f_{m-1} + \cdots + a_m^{(k)}f_0 \right|  
     \end{align*}
     for all $m$.
     
     We know that $a_0^{(k)}$ is bounded in $k$.  It is also easy to see that $|f_j| \leq \|f\|_p$ for all $j$.  If $a_0^{(k)}, a_1^{(k)},\ldots, a_j^{(k)}$ are also bounded in $k$, then the relation
     \[
          \left|a_{j+1}^{(k)}\right| \leq  \frac{1}{K^{m/r}|f_0|} + \left|\frac{a_0^{(k)}f_{j+1}}{f_0}  + \frac{a_1^{(k)}f_{j}}{f_0} +\cdots  +\frac{a_j^{(k)}f_1}{f_0}\right|
     \]
ensures that $a_{j+1}^{(k)}$  is bounded as well.  This proves that all of the coefficients of $Q_k$ are uniformly bounded in $k$.

Arguing as before, we may find a subsequence from $\{Q_k\}$ that converges uniformly, and the limit must be $q_{n,p}[f]$.   In fact, this must be the limit of the original sequence.
\end{proof}

\subsection{Roots of OPAs}

As a corollary to the last continuity theorem, we begin this subsection with two results concerning the set of possible OPA roots.
Let us first establish some notation. 
\begin{Definition}
For $1<p<\infty$ and $n\ge 0$, we denote the set of possible roots of OPAs of degree $n$ in $H^p$ as
\[
\Omega_{n ,p} := \left\{ w \in \mathbb{C} : \exists f \in H^p, f(0) \neq 0 \ \text{with} \ q_{n,p}[f](w) = 0 \right\}, 
\]
and let 
\[
\Omega_p := \bigcup_{n \ge 0} \Omega_{n, p}.
\]
\end{Definition}

Note that $\Omega_{0, p} =\emptyset$ for all $p \in (1, \infty)$. 
We have an immediate proposition concerning these sets. 
\begin{Proposition}\label{1-equals-p}
For $1<p<\infty$ and each $n\ge 1$, we have $\Omega_{n, p} \subseteq \Omega_{1, p}$, and therefore 
\[
\Omega_p = \Omega_{1,p}. 
\]
\end{Proposition}
\begin{proof}
    Suppose $w \in \Omega_{n, p}$ with $q_{n, p}[f](w) = 0$. Put $q_{n, p}[f] = (z-w)\tilde{q}$. Then, by optimality, we have
\begin{align*}
    \|q_{1,p}[\tilde{q}f]\tilde{q}f - 1\|_p
    &\le \|(z-w)\tilde{q}f - 1\|_p \\
    &= \|q_{n,p}[f]f - 1\|_p \\
    &\le \|q_{1,p}[\tilde{q}f]\tilde{q}f - 1\|_p,
\end{align*}
and we deduce that $q_{1,p}[\tilde{q}f] = q_{n,p}[f]/\tilde{q} = z-w$, which implies that $w \in \Omega_{1, p}$. 
\end{proof}

Presently, we see that the set of OPA roots must contain the set $\mathbb{C} \setminus \overline{\mathbb{D}}$.

\begin{Proposition}\label{roots}
    Let $1<p<\infty$.  If $w \in \mathbb{C} \setminus \overline{\mathbb{D}}$, then there exists $f \in H^p$ such that
    $q_{1,p}[f]$ has the root $w$, and so 
    \[
    \mathbb{C} \setminus \overline{\mathbb{D}} \subseteq \Omega_p.
    \]
\end{Proposition}

\begin{proof}
Let $w \in \mathbb{C} \setminus \overline{\mathbb{D}}$ and let  
\[
      f(z) := \frac{1}{z - w}, 
\]
which belongs to $H^p$ for all $p \in (1, \infty)$. 
Further,
\[
      \| 1 - Q(z) f(z) \|_p = 0
\]
when
\[
     Q(z) = z - w.
\]
Therefore, it must be that $q_{n,p}[f](z) = z - w$ for all $n \geq 1$. Hence, $w \in \Omega_p$.
\end{proof}

We now show that this set is connected and symmetric under rotation. 

\begin{Proposition}\label{connected}
     For $1<p<\infty$, the set $\Omega_p$ is rotationally symmetric and connected.
\end{Proposition}

\begin{proof}
    We begin by establishing rotational symmetry.  Let $f \in H^p$, and suppose $q_{1,p}[f] = a(z-w)$ (by Proposition \ref{1-equals-p}, it suffices to take the linear OPA).   
    Then for any $\gamma$ with $|\gamma| = 1$,
    \begin{align*}
         \|1 - a(z-w)\|_p^p  &=  \int_0^{2\pi} |1 - a(z-w)f(z)|^p\,dm(z) \\
              &=  \int_0^{2\pi} |1 - a(\gamma \zeta -w)f(\gamma\zeta)|^p|\gamma|^p\,dm(\zeta) \\
              &=   \int_0^{2\pi} |1 - a(\gamma \zeta -w)f(\gamma\zeta)|^p\,dm(\zeta) \\
              &=   \int_0^{2\pi} |1 - (a\gamma)(\zeta -\overline{\gamma}w)f(\gamma\zeta)|^p\,dm(\zeta).
    \end{align*}
    It must be that $(a\gamma)(z -\overline{\gamma}w)$ is the linear OPA for $f(\gamma z)$, for otherwise, by reversing these steps from
    \[
           \left\|1 - q_{1,p}[f(\gamma z)] f(\gamma z)\right\|_p^p
    \]
    we obtain a contradiction.
    
    This shows that if $w$ is an OPA root, then so is $\overline{\gamma} w$ for all $\gamma$, $|\gamma|=1$.  That is, the set $\Omega_p$ is rotationally symmetric.
    
Next, suppose that $f$ and $g$ belong to $H^p$, with real coefficients, and with $f(0)>0$ and $g(0)>0$.  Let their linear OPA roots be $r$ and $R$ respectively, where $0<r<R$.  By the continuity of the map $F\longmapsto q_{1,p}[F]$, we see that the set of linear OPA roots of the collection of functions $tf + (1-t)g$, $0 \leq t \leq 1$, must be an interval containing $[r,R]$; this is because the collection of functions is connected, and continuous maps preserve connectivity.  Note that $tf(0) + (1-t)g(0) > 0$ for all $t$, as required for the linear OPA to be nontrivial.  Consequently, $\Omega_p$ is path connected, and hence connected.
\end{proof}

\subsection{Fixed point approach}
Again, we mention that computing OPAs when $p \neq 2$ is a challenging task. Here, we explore the idea of OPAs being fixed points of an iterative process. 
We begin with the degree zero case and then move to the degree one case. 

\begin{Theorem}\label{fp-zero}
Let $2<p<\infty$, and let $f \in H^p$ be a nonconstant function.   Then the degree zero OPA $q_{0,p}[f]$ is the unique solution to the fixed point equation   
\[
      \zeta = \Phi(\zeta),
\]
where $\Phi: \mathbb{C} \longmapsto \mathbb{C}$ is given by
\[
      \Phi(\zeta) :=  \left(\int_\T|1 - \zeta f|^{p-2} \overline{f}\,dm  \right)\left( \int_\T |1 - \zeta f|^{p-2} |f|^2\,dm \right)^{-1}.
\]
Moreover, for any $\lambda_1 \in \mathbb{C}$, the sequence $\{\lambda_k\}$ given by $\lambda_{k+1} = \Phi(\lambda_k)$ converges to $q_{0,p}[f]$.
\end{Theorem}

\begin{proof}  Since $p>2$, and
\[
     1 = \frac{2}{p} + \frac{p-2}{p},
\]
the parameters $p/2$ and $p/(p-2)$ are H\"{o}lder conjugates of each other.  Hence H\"{o}lder's inequality gives
\[
     \int_\T |1 - \zeta f|^{p-2} |f|^2\,dm  \leq  \left(\int_\T |1 - \zeta f|^{(p-2)p/(p-2)} \,dm\right)^{(p-2)/p}\left(\int _\T|f|^{2(p/2)}\,dm\right)^{2/p}<\infty.
\]
Furthermore, since $f$ is nonconstant, the integral in the denominator of $\Phi$ is nonzero for any value of $\zeta$.  

When $\zeta \neq 0$, we have
\begin{align*}
    {|\zeta|} \int_\T |1 - \zeta f|^{p-2} |{f}| \,dm &\leq \int_\T |1 - \zeta f|^{p-2} \left( |1 - \zeta{f}| + 1\right)\,dm \\
      &=  \int_\T |1 - \zeta f|^{p-1}\,dm + \int_\T |1 - \zeta f|^{p-2} \,dm \\
      &< \infty;
\end{align*}
  and when $\zeta = 0$ 
  \[
       \int_\T |1 - \zeta f|^{p-2} |{f}| \,dm  = \int_\T |1 - 0\cdot  f|^{p-2} |{f}| \,dm  = \int_\T  |{f}| \,dm  < \infty.
  \] 
This verifies that $\Phi$ is well defined for all $\zeta \in \mathbb{C}$.  

In fact, $\Phi$ is continuous and bounded.  Continuity of the numerator and denominator of $\Phi$ at any point $\zeta_0$ can be established by a Dominated Convergence argument, with respective dominating functions 
\[
       2^{p-2} (1 + C|f|^{p-2})|f|\ \ \ \mbox{and}\ \ \ 2^{p-2} (1 + C|f|^{p-2})|f|^2,
\]
where $C > |\zeta_0|^{p-2}$.  Continuity at infinity is established by 
\begin{align*}
    \lim_{\zeta\rightarrow\infty} \frac{ \int_\T|1 - \zeta f|^{p-2} \overline{f}\,dm  }{ \int_\T|1 - \zeta f|^{p-2} |f|^2\,dm }
         &=   \lim_{\zeta\rightarrow\infty} \frac{ \int_\T|1/\zeta - f|^{p-2} \overline{f}\,dm  }{ \int_\T|1/\zeta - f|^{p-2} |f|^2\,dm } \\
        &=   \lim_{\zeta\rightarrow\infty} \frac{ \int_\T| f|^{p-2} \overline{f}\,dm  }{ \int_\T|f|^{p}\,dm }. 
\end{align*}

Consequently, $\Phi$ is a bounded function.  For any choice of $\lambda_1 \in \mathbb{C}$, define $\lambda_{k+1} = \Phi(\lambda_k)$ for all $k=1, 2, 3,\ldots$.  The resulting sequence $\{\lambda_k\}$ is a bounded sequence, and must contain a convergent subsequence, $\{\lambda_{n_k}\}$, with $\lambda_k \longrightarrow \lambda \in \mathbb{C}$.    Continuity ensures that 
\[
        \lambda = \Phi(\lambda),
\]
which is to say that
\begin{align*}
      \lambda \ \int_\T |1 - \lambda f|^{p-2} f \overline{f}\,dm  &= \int_\T |1 - \lambda f|^{p-2} \overline{f}\,dm\\
      0 &=  \int_\T |1 - \lambda f|^{p-2}(1 - \lambda {f}) \overline{f}\,dm,
\end{align*}
   or $1 - \lambda f \  \perp_p  \ f$.  This shows that $\lambda = q_{0,p}[f]$.   
   
But any subsequence of $\{\lambda_k\}$ must have a further subsequence that converges to the same limit.  Thus the sequence $\{\lambda_k\}$ itself must converge to $\lambda = q_{0,p}[f]$.
\end{proof}

We now discuss the linear case, first recording some notation.

Let a linear polynomial $Q_1(z) = a_1 + b_1 z$ be given and, for $k\geq 1$, let
\begin{align}
       \left[\begin{array}{c} a_{k+1} \\ b_{k+1} \end{array}\right] &= \left[ \begin{array}{cc} C_k & \overline{D_k} \\D_k & C_k   \end{array} \right]^{-1}
       \left[\begin{array}{c}  A_k \\ B_k \end{array}\right]  \label{recursivescheme2} \\
       & \nonumber  \\ 
             &= \frac{1}{|C_k|^2 + |D_k|^2}\left[ \begin{array}{rr} C_k & -\overline{D_k} \\ -D_k & C_k   \end{array} \right]
       \left[\begin{array}{c}  A_k \\ B_k \end{array}\right]  \nonumber  \\
       &  \nonumber \\
       &= \frac{1}{|C_k|^2 + |D_k|^2}
       \left[\begin{array}{c}  A_kC_k-B_k\overline{D_k} \\ B_kC_k-A_kD_k \end{array}\right],  \nonumber 
\end{align}
where
\begin{align}
    A_k &=  \int_\T |1 - Q_k f|^{p-2}\overline{f}\,dm  \label{formfora2} \\
    B_k &=    \int_\T |1 - Q_k f|^{p-2}\overline{zf}\,dm \nonumber  \\
    C_k &=    \int_\T |1 - Q_k f|^{p-2}|f|^2\,dm  \nonumber   \\
    D_k &=         \int_\T |1 - Q_k f|^{p-2}\overline{z} |f|^2\,dm, \nonumber 
\end{align}
and $Q_k(z) = a_k + b_k z$.  This determines a sequence of linear polynomials.

\begin{Theorem}\label{fp-one}
     Let $2<p<\infty$, and suppose that $f \in H^p$ is a nonconstant polynomial with $f(0) \neq 0$.  If $Q_k(z) = a_k + b_k z$, $k \geq 0$,  is the sequence of linear polynomials arising from \eqref{recursivescheme2}, then $Q_k$ converges to $q_{1,p}[f]$.
\end{Theorem}

\begin{proof}  If $Q_1$ is identically zero, then by inspection we see that $Q_2$ is not the zero polynomial.  Thus, by relabeling if necessary, let us assume $Q_1$ is not identically zero.

By the hypotheses on $f$, the expression $w-Qf$ is a nonconstant polynomial for any complex number $w$ and linear polynomial $Q$; hence
$|w - Qf|^{p-2}$ will be integrable on the unit circle.

Consider the expression, with integrals being taken over the circle,
\[
    \Phi(Q) :=   \left| \frac{ \int |Qf|^{p-2} \overline{f}\,dm  \int |Qf|^{p-2} |f|^2\,dm - \int |Qf|^{p-2} \overline{zf}\,dm  \int |Qf|^{p-2} \overline{z}|f|^2\,dm  }{ \left|\int |Qf|^{p-2} |f|^2\,dm \right|^2 + \left|\int |Qf|^{p-2} \overline{z}|f|^2\,dm \right|^2  }\right|,
\]
as $Q$ varies over set 
\[
\mathscr{Q} := \{ a + bz \in \Pol_1 : \max\{|a|, |b|\} = 1 \}.
\]
Under the assumptions on $f$, the denominator is bounded away from zero.  Thus $\Phi(Q)$ is a continuous function on a compact set, and achieves its maximum.  
In fact, the value of $\Phi(Q)$ is indifferent to rescaling $Q$, except for multiplying it by zero.  

From this we can further deduce that the values of
\begin{align*}
    &\Psi(Q,w) := \\  &\left| \frac{ \int |w-Qf|^{p-2} \overline{f}\,dm  \int |w-Qf|^{p-2} |f|^2\,dm - \int |w-Qf|^{p-2} \overline{zf}\,dm  \int |w-Qf|^{p-2} \overline{z}|f|^2\,dm  }{ \left|\int |w-Qf|^{p-2} |f|^2\,dm \right|^2 + \left|\int |w-Qf|^{p-2} \overline{z}|f|^2\,dm \right|^2  }\right|,
\end{align*}
are uniformly bounded for $Q \in \mathscr{Q}$ and $|w| \leq 1$.

Next, notice that for any nonzero linear polynomial $Q(z) = a + bz$ we have
\[
      \int_\T |1 - Q f|^{p-2}\overline{f}\,dm  =  \int_\T |1 - (a + bz) f|^{p-2}\overline{f}\,dm =  |c|^{p-2}\int_\T |1/c - (a/c + [b/c]z) f|^{p-2}\overline{f}\,dm,
\]
where $c := \max\{|a|, |b|\}$.  This is to say that the value of $A_k$ in \eqref{formfora2} scales in a simple way with $c$, with the result that $f$ is multiplied by a member of $\mathscr{Q}$, and the $1$ inside the integrand is replaced by $1/c$.  Similar remarks apply to the formulas for $B_k$, $C_k$, and $D_k$.

Consequently, when $A_k$, $B_k$, $C_k$ and $D_k$ are assembled together to yield $a_{k+1}$ and $b_{k+1}$, the scaling factors $|c|^{p-2}$ attached to each integral cancel.   

Let us write $c_{k} := \max\{|a_{k}|, |b_{k}|\}$.
The above observations establish that $|c_{k+1}|$ is uniformly bounded as  $k$ varies over such indices that $|c_k| \geq 1$.

For the other values of $k$, for which $|c_k| < 1$, the corresponding expressions for $|1 - Q_k f|^{p-2}$ are again uniformly bounded in the obvious way, implying that the resulting $c_{k+1}$ are also uniformly bounded.

This shows that $\{Q_k\}$ is a bounded sequence of linear polynomials, which must therefore have a convergent subsequence.  The limit is a linear polynomial $Q_{\infty}$, which satisfies the orthogonality conditions for $q_{1,p}[f]$, and hence must be the OPA.  Uniqueness of the OPA ensures that, in fact, every subsequence of $\{Q_k\}$ has a further subsequence that converges to $q_{1,p}[f]$.  In conclusion, we have
\[
       \lim_{k\rightarrow\infty} Q_k = q_{1,p}[f].
\]
\end{proof}

We end this section by noting that Theorems \ref{fp-zero} and \ref{fp-one} are only established for $2 < p < \infty$, and, in the degree one case, only for polynomials. It is currently unclear if these results extend to $1 < p < 2$, or if analogous results hold for higher degree OPAs.



\section{Error Bounds and Duality Arguments}\label{error}

The present section is concerned with estimating (both above and below) the quantity $\|q_{n,p}[f]f - 1\|_p$, i.e., the ``error'' in the optimal polynomial approximation algorithm. 
We begin by employing some duality methods, first recalling a fundamental result from classical functional analysis, tailored to our setting. 

 \begin{Lemma}\label{duality}
 Let $1 < p < \infty$ and $f \in H^p$. For any $n\in \mathbb{N}$, we have
 \[
\| q_{n,p}[f]f -1\|_p  =   \left[ \inf_{\psi \in L^q} \left\{ \|\psi\|_{L^q}   : \psi_0 = 1, \langle z^k f, \psi \rangle =0\ \forall \ 0 \leq k \leq n  \right\} \right]^{-1}.
\]
 \end{Lemma} 
  
 \begin{proof} 
By an elementary duality theorem of functional analysis, with respect to the pairing
\[
       \langle f, g \rangle =  \int_0^{2\pi} f(e^{i\theta}) \overline{g(e^{i\theta})}\,\frac{d\theta}{2\pi},
\]
we have
\begin{align}
     \| q_{n,p}[f]f -1\|_p  &=  \inf \left\{ \|Qf -1\|_p:\ Q \in \mathscr{P}_n  \right\} \nonumber \\
       &= \mbox{dist}_{H^p}(1, \mathscr{P}_n f) \nonumber \\
       &= \mbox{dist}_{L^p}(1, \mathscr{P}_n f) \label{curlyfries} \\
       &= \|1\|_{[(\mathscr{P}_n f)^{\perp}]^*}  \label{chocolateshake}\\
       &= \sup \left\{ \frac{| \langle  \psi, 1 \rangle  |}{\|\psi\|_{L^q}}:\ \psi \in (\mathscr{P}_n f)^{\perp}\setminus\{0\}  \right\} \nonumber\\
       &= \sup \left\{ \frac{|\psi_0|}{\|\psi\|_{L^q}}:\ \psi \in (\mathscr{P}_n f)^{\perp}\setminus\{0\}  \right\} \nonumber \\
       &= \Big[\inf\left\{\|\psi/\psi_0\|_{L^q}:\ \psi \in (\mathscr{P}_n f)^{\perp}\setminus\{0\}\right\}   \Big]^{-1} \nonumber \\
       &= \Big[ \inf \left\{ \|\psi\|_{L^q}   :\   \psi \in (\mathscr{P}_n f)^{\perp},\ \psi_0 = 1  \right\} \Big]^{-1} \nonumber \\
       &= \Big[ \inf \left\{ \|\psi\|_{L^q}   :\   \psi \in L^q,\ \psi_0 = 1, \langle z^k f, \psi \rangle =0\ \forall 0 \leq k \leq n  \right\} \Big]^{-1}. \nonumber 
\end{align}
\end{proof}

\begin{Remark}

The reason to move to $L^p$ in line \eqref{curlyfries} is that the dual of $L^p$ is $L^q$.  If we stick with the norm in $H^p$, then (caution!) the relevant dual space is the quotient space $L^q/H^q$, rather than $H^q$.  These spaces are isometrically isomorphic only when $p=2$.  

In line \eqref{chocolateshake}, we mean ``the norm of the unit constant function, viewed as a bounded linear functional on the annihilator of the subspace spanned by $f\mathscr{P}_n$.''
\end{Remark}

We first apply this duality to provide a lower bound for the OPA error in the case that we are approximating a polynomial with zeros in the disk. 
\begin{Proposition}\label{poly-lower}
Suppose $f$ is a polynomial
\[
     f(z) = (z - w_1) (z - w_2)\cdots(z - w_d),
\]
with the roots being distinct, nonzero and contained inside $\mathbb{D}$. Then, for $1<p<\infty$ and $\lambda := q_{0,p}[f]$, we have
\[
     \|1 - \lambda f\|_p  \geq  1 - |w_1 w_2\cdots w_d|.
\]
\end{Proposition}

\begin{proof}
The space of functions in $L^q$ which annihilate $f$ contains functions of the form
\[
     c_1 \Lambda_{w_1} + c_2 \Lambda_{w_2} + \cdots + c_d \Lambda_{w_d},
\]
where $\Lambda_{w_j}$ denotes the point evaluation functional (or Szeg\"{o} kernel) at the point $w_j \in \D$. 
Thus, by the Lemma \ref{duality}, we have
\[
     \| 1 - \lambda f\|_p  =  \left[ \inf \|\psi\|_q   \right]^{-1},
\]
where the infimum is over $\psi\in L^q$ satisfying $\langle f,\psi \rangle = 0$ and $\psi_0 = 1$.

Let 
\[
     B(z) = a \prod_{k=1}^{d} \frac{w_k - z}{1 - \overline{w}_k z},
\]
a constant multiple of the Blaschke product with the same zeros as $f$.
Its numerator has leading term $\pm a z^d$, while the denominator has leading term $\pm \overline{w_1 w_2\cdots w_d} z^d$ (with matching signs).  Thus long division followed by partial fractions expansion results in an expression of the form
\[
     B(z) = \frac{a}{\ \overline{w_1 w_2\cdots w_d}\ } +  c_1 \Lambda_{w_1} + c_2 \Lambda_{w_2} + \cdots + c_d \Lambda_{w_d}.
\]

Evaluating this equation at $z=0$ tells us that 
\[
    a w_1 w_2\cdots w_d = \frac{a}{\ \overline{w_1 w_2\cdots w_d}\ } + c_1 + c_2 + \cdots + c_d.
\]

This suggests making the specific choice of
\[
   \psi(z) =    c_1 \Lambda_{w_1} + c_2 \Lambda_{w_2} + \cdots + c_d \Lambda_{w_d}
\]
with the coefficients determined above.
The requirement of  $\psi(0) = 1$ therefore gives
\[
     a w_1 w_2\cdots w_d = \frac{a}{\ \overline{w_1 w_2\cdots w_d}\ } + 1,
\]
which will furnish the value of $a$, namely,
\[
     a = \left[ w_1 w_2\cdots w_d -  \frac{1}{\ \overline{w_1 w_2\cdots w_d}\ } \right]^{-1}.
\]

Finally, an application of the triangle inequality yields
\begin{align*}
      \| 1 - \lambda f\|_p  &\geq \|\psi\|_q^{-1} \\
         &\geq \left\|B - \frac{a}{\ \overline{w_1 w_2\cdots w_d}\ } \right\|_q^{-1} \\
         &\geq \left[ \|B\|_q + \left\|\frac{a}{\ \overline{w_1 w_2\cdots w_d}\ } \right\|_q  \right]^{-1} \\
         &= \left[  |a| + \left| \frac{a}{w_1 w_2\cdots w_d}\right| \right]^{-1}\\
         &=  \frac{1 - |f(0)|^2}{1 + | f(0)|}\\
         & = 1 - |f(0)| \\
         &= 1 - |w_1 w_2\cdots w_d|.
\end{align*}
\end{proof}

\begin{Remark}
    The above result holds for any function $f$  vanishing at the points $w_1,\ldots, w_d$.  Further, by Theorem \ref{cont-f}, the result also extends to any infinite Blaschke product.
\end{Remark}

Let us now use duality to further investigate OPA errors for more general functions.

\begin{Proposition}\label{psi-lower}
Let $1<p<\infty$, $1/p +1/q =1$, $n \in \mathbb{N}$, and $f \in H^p$ with $f(0) = 1$.  Then
\[
\|q_{n-1,p}[f]f-1\|_p \geq   \frac{1}{\|1+\psi_1 z + \psi_2 z^2 + \cdots+ \psi_n z^n \|_q},
\]
where the coefficients $\psi_k$, $1 \leq k \leq n$, satisfy the matrix equation
\[
\left[
\begin{array}{ccccc}
f_1 & f_2 & f_3 & \cdots & f_n \\
f_0 & f_1 & f_2 & \cdots & f_{n-1}\\
0   &  f_0  & f_1 & \cdots & f_{n-2}\\
\vdots & \vdots & \vdots & \ddots & \vdots \\
0 & 0 & 0 & \cdots & f_1
\end{array}
\right]
\left[
\begin{array}{c}
\psi_1 \\ \psi_2 \\ \psi_3 \\ \vdots  \\ \psi_n   \end{array}
\right]
=
\left[
\begin{array}{r}
-1\\ 0 \\ 0 \\ \vdots \\ 0
\end{array}
\right].
\]
\end{Proposition}

\begin{proof}
It suffices to check that the function 
\[
\psi(z) = 1+\psi_1 z + \psi_2 z^2 + \cdots+ \psi_n z^n
\]
satisfies the hypotheses of Lemma \ref{duality}. That is, for $0 \le k \le n$, that 
\[
\langle z^k f, \psi \rangle =0.
\]
This is ensured precisely by the linear system in the statement of the proposition. 
\end{proof}

With further calculation, the approach in the previous proposition can be used to show the following:

\begin{Proposition}\label{snail}
     Let $1<p<\infty$, $1/p +1/q =1$, $n \in \mathbb{N}$, and $f \in H^p$ with $f(0) = 1$.   Let
     \[
         \frac{1}{f(z)} = 1 + g_1 z + g_2 z^2 + \cdots
     \]
     be the power series of $1/f$ about the origin.
     
   Then
  \[
     \|q_{n-1,p}[f]f-1\|_p \geq  \frac{|g_n|}{\left\|1+g_1 z + g_2 z^2 +\cdots+g_n z^n\right\|_q}.
   \]
\end{Proposition}

\begin{proof}

Let us begin with the matrix equation from Proposition \ref{psi-lower}:
\[
\left[
\begin{array}{ccccc}
f_1 & f_2 & f_3 & \cdots & f_n \\
f_0 & f_1 & f_2 & \cdots & f_{n-1}\\
0   &  f_0  & f_1 & \cdots & f_{n-2}\\
\vdots & \vdots & \cdots & \ddots & \vdots \\
0 & 0 & 0 & \cdots & f_1
\end{array}
\right]
\left[
\begin{array}{c}
\psi_1 \\ \psi_2 \\ \psi_3 \\ \vdots  \\ \psi_n   \end{array}
\right]
=
\left[
\begin{array}{r}
-1\\ 0 \\ 0 \\ \vdots \\ 0
\end{array}
\right].
\]
There is no harm in multiplying both sides of the equation on the left by the elementary permutation matrix
\[
\left[
\begin{array}{ccccc}
0 & 0 & 0 & \cdots & 1 \\
1 & 0 & 0 & \cdots & 0\\
0 &  1  & 0 & \cdots & 0\\
\vdots & \vdots & \vdots & \ddots & \vdots \\
0 & 0 & 0 & \cdots & 0
\end{array}
\right]
\]
(the next to last entry of the bottom row is 1), which has the effect of changing the equation to
\[
    \left[ \begin{array}{ccccc} 
        f_0 & f_1 & f_2 & \cdots & f_{n-1}\\
        0   &  f_0  & f_1 & \cdots & f_{n-2}\\
        0 &  0 & f_0 & \cdots & f_{n-3} \\
        \vdots & \vdots & \vdots & \ddots & \vdots \\
        0 & 0 & 0 & \cdots & f_1 \\
        f_1 & f_2 & f_3 & \cdots & f_n 
   \end{array}   \right]
   \left[\begin{array}{c} \psi_1 \\ \psi_2 \\ \psi_3 \\ \psi_4 \\ \vdots  \\ \psi_n   \end{array}\right]
   =
   \left[\begin{array}{r} 0\\ 0 \\ 0 \\ 0 \\ \vdots \\ -1 \end{array}\right].
\]

By successively subtracting multiples of the other rows, the bottom row can be placed in the form
\[
     \left[\begin{array}{ccccc} 0 & 0 & 0 & \cdots & C\end{array}\right]
\]
for some constant $C$, which could be zero.  In fact, recalling that $f_0 = 1$, we see that $C$ must be given by
\[
C =
\det
\left[
\begin{array}{ccccc}
f_1 & f_2 & f_3 & \cdots & f_n \\
f_0 & f_1 & f_2 & \cdots & f_{n-1}\\
0   &  f_0  & f_1 & \cdots & f_{n-2}\\
\vdots & \vdots & \vdots & \ddots & \vdots \\
0 & 0 & 0 & \cdots & f_1
\end{array}
\right].
\]

Furthermore, the sequence of row operations to diagonalize the matrix leaves the right side unchanged as
\[
      \left[\begin{array}{ccccc} 0 & 0 & 0 & \cdots & -1\end{array}\right]^T.
\]

Assuming that $C \neq 0$, and again recalling that $f_0 = 1$, our matrix equation can be written as
\[
    \left[ \begin{array}{ccccc} 
        f_0 & f_1 & f_2 & \cdots & f_{n-1}\\
        0   &  f_0  & f_1 & \cdots & f_{n-2}\\
        0 &  0 & f_0 & \cdots & f_{n-3} \\
        \vdots & \vdots & \vdots & \ddots & \vdots \\
        0 & 0 & 0 & \cdots & f_0 
   \end{array}   \right]
   \left[\begin{array}{c} \psi_1 \\ \psi_2 \\ \psi_3 \\  \vdots  \\ \psi_n   \end{array}\right]
   =
   \left[\begin{array}{r} 0\\ 0 \\ 0 \\ \vdots \\ -1/C \end{array}\right].
\]

The inverse of the transposed (Toeplitz) matrix on the left is simply
\[
    \left[ \begin{array}{ccccc} 
        g_0 & g_1 & g_2 & \cdots & g_{n-1}\\
        0   &  g_0  & g_1 & \cdots & g_{n-2}\\
        0 &  0 & g_0 & \cdots & g_{n-3} \\
        \vdots & \vdots & \vdots & \ddots & \vdots \\
        0 & 0 & 0 & \cdots & g_0 
   \end{array}   \right],
\]
where $g(0)=1$ and $g(z) =  g_0 + g_1 z + g_2 z^2+\cdots$ is the Taylor expansion of $1/f(z)$, valid for some disk centered at the origin.
The conclusion is that 
\[
       \psi_k = -g_{n-k}/C\ \ \forall 0\leq k\leq n-1.
\]

Our next challenge is to find an analytical meaning for the constant $C$.  But notice that the row operations needed to clear entries from the bottom row of
\begin{equation}\label{milkduds}
 \left[ \begin{array}{ccccc} 
        f_0 & f_1 & f_2 & \cdots & f_{n-1}\\
        0   &  f_0  & f_1 & \cdots & f_{n-2}\\
        0 &  0 & f_0 & \cdots & f_{n-3} \\
        \vdots & \vdots & \vdots & \ddots & \vdots \\
        0 & 0 & 0 & \cdots & f_1 \\
        f_1 & f_2 & f_3 & \cdots & f_n 
   \end{array}   \right]
\end{equation}
would (suitably modified) similarly clear the second through the last entries from the top row.  Performing all of these (suitably modified) row operations on the identity matrix would have to result in
\[
       \left[ \begin{array}{ccccc} 
        1 & g_1 & g_2 & \cdots &  g_{n-1}\\
        0   &  1  & 0 & \cdots &  0\\
        0 &  0 & 1 & \cdots   & 0 \\
        \vdots & \vdots & \vdots & \ddots & \vdots \\
        0 & 0 & 0 & \cdots & 1
   \end{array}   \right]
\]

Then following carefully what operations are correspondingly performed on the last column in \eqref{milkduds}, we conclude that
\[
     C = f_n + g_1 f_{n-1} + g_2 f_{n-2} + \cdots+ g_{n-1} f_1 = -g_n.
\]

Finally, note that 
\begin{align*}
      \left\|1+ \frac{g_{n-1}}{g_n} z + \frac{g_{n-2}}{g_n} z^2 + \cdots+ \frac{1}{g_n} z^n \right\|_q &= 
      \frac{1}{|g_n|}\left\|g_n+ g_{n-1} z + g_{n-2} z^2 + \cdots+ z^n \right\|_q\\
      &=   \frac{1}{|g_n|}\left\|z^{-n}(g_n+ g_{n-1} z + g_{n-2} z^2 + \cdots+ z^n) \right\|_q\\     
      &=  \frac{1}{|g_n|}\left\|1+g_1 z^{-1} + g_2 z^{-2} +\cdots+g_n z^{-n}\right\|_q\\
      &=  \frac{1}{|g_n|}\left\|1+{g}_1 z + {g}_2 z^{2} +\cdots+{g}_n z^{n}\right\|_q,
\end{align*}
where in the last step, the change of variable $\theta \longmapsto -\theta$, for $z = e^{i\theta}$, leaves the norm integral net unchanged.
\end{proof}

We will provide an improvement to the above proposition, but we must first consider the problem of finding $G \in zH^p$ such that
\[
    \left\|\overline{G} + F\right\|_p
\]
is minimized.  Duality tells us that (as we continue to mark extremal functions with (*))
\begin{align*}
     \left\|\overline{G^*} + F\right\|_p &= \sup\left\{ \frac{|\langle F, \psi\rangle|   }{\|\psi\|_q  }:\ \psi \in H^q\right\}\\
       &= \frac{|\langle F, \psi^*\rangle|   }{\|\psi^*\|_q  }\\
       &= |\langle F, \psi^*\rangle|/\inf\left\{\| \psi^* + K\|_q:\ K \in \overline{zH^q}\right\}.
\end{align*}
Once again, we are up against the dual of $H^p$ being isomorphic to $H^q$, but not isometrically.

Nonetheless, we must consider the metric projection of $F$ onto the subspace $\overline{zH^p}$.  Notice that $(\overline{G^*} + F)^{\langle p-1\rangle}$ annihilates any negative frequencies.  Therefore, there exists $h \in H^q$ such that
\[
(\overline{G^*} + F)^{\langle p-1\rangle} = h,
\]
and thus, taking a $\langle q-1 \rangle$ power, we have
\[
|h|^{q-2} h = \overline{G^*} + F.
\] 
In turn, we see that finding $G^*$ amounts to solving the above highly unpleasant functional equation.

Let us record this in the following result, where we write $P_+$ for the Riesz projection, given by 
\[
       \sum_{k = -\infty}^{\infty} c_k z^k  \longmapsto  \sum_{k = 0}^{\infty} c_k z^k, 
\]
which is bounded from $L^p \to H^p$.  

\begin{Proposition}
   Let $1<p<\infty$ and $1/p+1/q = 1$. Suppose $h \in H^q$, and define $F := P_+\overline{h^{\langle q-1\rangle}}$.  Then
   \[
         \inf \left\{\|F + \overline{G}\|_p:\ G \in zH^p \right\}
   \]
   is attained by taking $F + \overline{G} = \overline{h^{\langle q-1\rangle}}$.  In this case, the value of the infimum is given by
   \[
         \inf \left\{\|F + \overline{G}\|_p:\ G \in zH^p \right\} = \|h\|_q^{q-1}.
   \]
\end{Proposition}

This warrants the following observation:

\begin{Proposition} For $1<p<\infty$, the set of images
$P_+\, \overline{(H^q)^{\langle q-1\rangle}}$ is dense in $H^p$.
\end{Proposition}

\begin{proof}
Suppose $g \in H^q$ has the property that
\[
\left\langle P_+\, \overline{h^{\langle q-1\rangle}}, g \right\rangle = 0 
\]
for all $h \in H^q$.  Then
\begin{align*}
     0 &= \left\langle P_+\, \overline{h^{\langle q-1\rangle}}, g \right\rangle\\
       &= \int_{\mathbb{T}} P_+\, \overline{h^{\langle q-1\rangle}} \overline{g}\, dm\\
       &= \int_{\mathbb{T}} \overline{h^{\langle q-1\rangle}} \overline{g}\, dm.
\end{align*}
We are able to drop the projection in the last line since integration against $\overline{g}$ will annihilate the negative frequencies of $\overline{h^{\langle q-1\rangle}}$.
In particular, this must hold for $h=g$, hence $0 = \|g\|_q^q$.   This forces $g$ to be identically zero.
\end{proof}

In turn, we can make the following improvement to Proposition \ref{snail}.

\begin{Proposition}
 Let $1<p<\infty$, $1/p +1/q =1$, $n \in \mathbb{N}$, and $f \in H^p$ with $f(0) = 1$.   Let
     \[
         \frac{1}{f(z)} = 1 + g_1 z + g_2 z^2 + \cdots
     \]
     be the power series of $1/f$ about the origin.
     
   If $g_n \neq 0$, then
\[
     \|q_{n-1,p}[f]f-1\|_p  \geq  \frac{1}{\left\|h^{\langle p-1\rangle}\right\|_q},
\]
where $h\in H^p$ satisfies
\[
          P_+ \, \overline{h^{\langle p-1\rangle}} =  1+\frac{g_{n-1}}{g_n} z + \frac{g_{n-2}}{g_n} z^2 +\cdots+\frac{g_0}{g_n} z^n.
\]

\end{Proposition}

Let us now consider the case where $n \longrightarrow \infty$.  Then, by writing $f = JG$ for $J$ inner and $G$ outer, we see, for all $k \ge0$, that $\psi \in L^q$ satisfies
\begin{align*}
    0 &=  \langle z^k f, \psi \rangle \\
    &=  \langle z^k JG, \psi \rangle \\
    &=  \langle z^k G, \overline{J}\psi \rangle.
\end{align*}
As $\left\{z^kG : k\ge 0 \right\}$ is dense in $H^p$, we have, for any $n\ge 0$, 
\begin{align*}
   0 &= \langle z^n, \overline{J}\psi \rangle \\
    &=  \int_0^{2\pi} J(e^{i\theta}) \overline{\psi(e^{i\theta})} e^{in\theta}\,\frac{d\theta}{2\pi}.
\end{align*}

From this, it follows that $K(z) := J(z)\overline{\psi(z)}/z$ is an element of $H^q$.  We further divine that $\psi$ must be determined by
\[
     \psi(z) = J(z) \overline{zK(z)},\ z \in \mathbb{T},
\]
for some $K \in H^q$.  The condition $\psi(0) = 1$ takes the form
\[
      1=J_1 \overline{K_0} + J_2 \overline{K_1}+J_3 \overline{K_2}+\cdots.
\]

We now must minimize $\|J \overline{zK}\|_q$ subject to $K \in H^q$ satisfying the above constraint.
It is tempting to try $K = \overline{J}$, but this will not work, since 
\[
     0 = J_1 \overline{J_0} + J_2 \overline{J_1}+J_3 \overline{J_2}+\cdots.
\] 

Instead, take $K(z) = c [J(z)-J(0)]/z,$ where $c^{-1} = |J_1|^2 + |J_2|^2 +|J_3|^2+\cdots$.   Then 
\[
   J_1 \overline{K_0} + J_2 \overline{K_1}+J_3 \overline{K_2}+\cdots = c\left(|J_1|^2 + |J_2|^2 +|J_3|^2+\cdots\right) = 1,
\]
as needed.

Using this choice of $K$ to compute $\psi$, we obtain 
\begin{align*}
      \psi(z) &= J(z) \overline{zK(z)} \\
       &=  cJ(z) \overline{z} \ \frac{\overline{J(z)} - \overline{J(0)}}{\overline{z}}\\
       &=  c(1 - J(z) \overline{J(0)}).
\end{align*}

Since $ \|q_{n,p}[f]f-1\|_p \geq {1}/{\|\psi\|_q}$, this furnishes the following bound.

\begin{Proposition}
Let $1<p<\infty$, $1/p +1/q =1$, $n \in \mathbb{N}$, and $f \in H^p$ with $f(0) = 1$.  Then
\[
     \|q_{n,p}[f]f-1\|_p \geq   \frac{|J_1|^2 + |J_2|^2 +|J_3|^2+\cdots}{\|1- \overline{J(0)}J(z)\|_q},
\]
where $J$ is the inner part of $f$.
\end{Proposition}

Incidentally, $c^{-1} = \left(|J_1|^2 + |J_2|^2 +|J_3|^2+\cdots\right) = \|J\|_2^2 - |J(0)|^2 = 1 - |J(0)|^2$, so the lower bound above could be written equivalently as
\[
     \frac{1 - |J(0)|^2}{\left\|(1 - |J(0)|^2)  - \overline{J(0)}(J_1 z + J_2 z^2 + \cdots)\right\|_q},
\]
which is obviously no greater than 1, as needed.

We now step away from duality. Our final results concern OPA errors, but are proven with $H^2$ methods.
The following proposition should be compared with Proposition \ref{poly-lower}; although the result below provides a better bound, it holds only for $p>2$. 

\begin{Proposition}
   Let $2<p<\infty$, and suppose $f \in H^p$ has a factorization $f = JG$, where $J$ is inner and $G$ is outer.  If $f(0)\neq 0$, then
   for any $n \in \mathbb{N}$, 
   \[
         \| q_{n,p}[f] f -1\|_p  \geq \sqrt{1 - |J(0)|^2}.
   \]
\end{Proposition}

\begin{proof}  Let $\mathscr{P}$ be the collection of all polynomials. Then
   \begin{align*}
       \| q_{n,p}[f] f -1\|_p &\geq  \inf_{Q \in \mathscr{P}}\|Qf -1\|_p\\
          & \geq  \inf_{Q \in \mathscr{P}}\|Qf -1\|_2\\
          & =  \inf_{Q \in \mathscr{P}}\|QJG -1\|_2\\
          & =  \inf_{Q \in \mathscr{P}}\|QG -\overline{J}\|_2\\
          & =  \|\overline{J(0)} -\overline{J}\|_2,
   \end{align*}
with the last equality following from the fact that since $G$ is outer, then the set $\{QG : Q\in \Pol\}$ is dense in $H^2$. 
Now use 
\[
     1 = \|J\|_2^2 = |J(0)|^2 + \|J - J(0)\|_2^2.
\]
\end{proof}

Note further that if $J=B$ is a Blaschke product, then this implies that
 \[
         \| q_{n,p}[f] f -1\|_p  \geq \|{J(0) - J}\|_2 = \sqrt{1 - |B(0)|^2} = \sqrt{1 - |w_1 w_2 w_3\cdots|^2},
   \]
where $w_1, w_2, w_3,\ldots$ are the zeros of $B$.

We end by providing a related result when $1<p<2$.

\begin{Proposition}
Let $1<p<2$ and suppose  $f(0)\neq 0$. Then for any $n \in \mathbb{N}$, 
\[
\| q_{n,p}[f] f -1\|_p  \leq \sqrt{1 - (q_{n,2}[f]f)(0)}.
\]
\end{Proposition}

\begin{proof}  Routine bounds yield
   \begin{align*}
      \| q_{n,p}[f] f -1\|_p  &\leq  \| q_{n,2}[f] f -1\|_p\\
         &\leq \| q_{n,2}[f] f -1\|_2 \\
         & = \sqrt{1 - (q_{n,2}[f]f)(0)},
   \end{align*}
where the last equality is a consequence of the linear system described in Equation \ref{opt-sys}.
\end{proof}
Taking $n = 0$ above, we have the simple bound
\[
\| q_{n,p}[f] f -1\|_p  \leq \left(1 - \frac{|f(0)|^2}{\|f\|_2^{2}}\right)^{1/2}.
\]




\bibliographystyle{abbrv}

\bibliography{OPAs_ACT_II_FINAL.bib}


\end{document}